\documentclass[a4paper,reqno,10pt]{amsart}

\usepackage{epsf,graphicx,epsfig}
\usepackage{amscd}
\usepackage{amsmath,latexsym,amssymb,amsthm}
\usepackage[nospace,noadjust]{cite}
\usepackage{textcomp}
\usepackage{setspace,cite}
\usepackage{lscape,fancyhdr,fancybox}
\usepackage{stmaryrd}
\usepackage[all,cmtip]{xy}
\usepackage{tikz}
\usepackage{cancel}
\usetikzlibrary{shapes,arrows,decorations.markings}
\usepackage{graphicx}

\usepackage{color}
\usepackage{url}
\usepackage{enumerate}
\usepackage[mathscr]{euscript}

  \theoremstyle{plain}

\swapnumbers
    \newtheorem{thm}{Theorem}[section]
    \newtheorem{prop}[thm]{Proposition}
     
   \newtheorem{lemma}[thm]{Lemma}
    \newtheorem{corollary}[thm]{Corollary}
    
    \newtheorem{subsec}[thm]{}
\theoremstyle{definition}
    \newtheorem{defn}[thm]{Definition}
        \newtheorem{remark}[thm]{Remark}
    \newtheorem{exam}[thm]{Example}

\theoremstyle{remark}

\title{}
\author{}
\date{}
\usepackage{amssymb}

\usepackage{hyperref}
\hypersetup{
	colorlinks,
	citecolor=blue,
	filecolor=black,
	linkcolor=blue,
	urlcolor=black
}

\begin{document}

\title{On $\mathcal{O}$-operators on modules over Lie algebras}

\author{Apurba Das}

\maketitle

\begin{center}
Department of Mathematics and Statistics,\\
Indian Institute of Technology, Kanpur 208016, Uttar Pradesh, India.\\
Email: apurbadas348@gmail.com
\end{center}



\begin{abstract}
The notion of $\mathcal{O}$-operators on modules over Lie algebras generalize Rota-Baxter operators. They also generalize Poisson structures on Lie algebras in the presence of modules. Motivated from Poisson structures, we define gauge transformations and reductions of $\mathcal{O}$-operators. Next we consider compatible $\mathcal{O}$-operators on modules over Lie algebras. We define $\mathcal{ON}$-structures which give rise to hierarchy of compatible $\mathcal{O}$-operators. We show that a solution of the strong Maurer-Cartan equation on a twilled Lie algebra associated to an $\mathcal{O}$-operator gives rise to an $\mathcal{ON}$-structure, hence, a hierarchy of compatible $\mathcal{O}$-operators. Finally, we also introduce generalized complex structures and holomorphic $\mathcal{O}$-operators on modules over Lie algebras and show how they incorporate $\mathcal{O}$-operators.\\
\end{abstract}



\thispagestyle{empty}


\vspace{0.2cm}

\section{Introduction}
The notion of Rota-Baxter operators on associative algebras was introduced by G. Baxter \cite{baxter} and G.-C. Rota \cite{rota} in 1960's in their study of the fluctuation theory of probability and combinatorics. In last twenty years, Li Guo made significant contributions in Rota-Baxter algebras. See for instance \cite{guo-keigher,guo-book}. More precisely, a Rota-Baxter operator (of weight $0$) on an associative algebra $A$ is a linear map $R : A \rightarrow A$ that satisfies $R(a)R(b) = R ( R(a)b + a R(b))$, for $a, b \in A$. Rota-Baxter operators are algebraic abstraction of integral operators. An importance of these operators are shown by Connes and Kreimer in the algebraic approach of renormalization in quantum field theory \cite{connes-kreimer}. Such operators are also useful in the study of dendriform algebras and splitting of operads \cite{aguiar}. Rota-Baxter operators can also defined in a Lie algebra \cite{bai,bai-guo-ni}. Let $(\mathfrak{g},[~,~])$ be a Lie algebra. A linear map $R : \mathfrak{g} \rightarrow \mathfrak{g}$ is called a Rota-Baxter operator (of weight $0$) if $R$ satisfies
\begin{align*}
[R(x), R(y)] = R ( [R(x), y] - [R(y), x]), ~ \text{ for } x, y \in \mathfrak{g}. 
\end{align*}

\medskip

The notion of generalized Rota-Baxter operators on bimodules over associative algebras was introduced by K. Uchino motivated from Poisson structures \cite{uchino}. Their Maurer-Cartan characterizations, cohomology and deformation theory are studied in \cite{das}. 
Generalized Rota-baxter operators in the context of Lie algebras was previously appeared in the work of Kupershmidt by the name of $\mathcal{O}$-operators \cite{kuper}. Let $(\mathfrak{g}, [~,~])$ be a Lie algebra and $(M, \bullet)$ be $\mathfrak{g}$-module. A linear map $T : M \rightarrow \mathfrak{g}$ is called an $\mathcal{O}$-operator on $M$ over $\mathfrak{g}$ if it satisfies
\begin{align*}
[T(m), T(n)] = T ( T(m) \bullet n - T(n) \bullet m), ~ \text{ for }m, n \in M. 
\end{align*}
It turns out that $M$ carries a Lie algebra structure with bracket $[m,n]^T := T(m) \bullet n - T(n) \bullet m.$

\medskip

In this paper, we study $\mathcal{O}$-operators in the context of Lie algebras from Poisson geometric perspectives. In Section \ref{sec-prelim} we first recall the Chevalley-Eilenberg cohomology (CE cohomology) of Lie algebras and Nijenhuis operators.
Section \ref{sec-o-op} begins with $\mathcal{O}$-operators. Given an $\mathcal{O}$-operator $T$ and a `suitable' $1$-cocycle $B: \mathfrak{g} \rightarrow M$ in the CE cohomology of $\mathfrak{g}$ with coefficients in $M$, we construct a new $\mathcal{O}$-operator $T_B : M \rightarrow \mathfrak{g}$, called the gauge transformation of $T$ by $B$. This construction is inspired from the gauge transformation of Poisson structures introduced by \v{S}evera and Weinstein \cite{sev-wein} (see also \cite{das0}). The $\mathcal{O}$-operators $T$ and $T_B$ induce isomorphic Lie algebra structures on $M$ (Proposition \ref{iso-m-mt}). In the next, we generalize the Marsden-Ratiu Poisson reduction theorem \cite{mars-ratiu} to $\mathcal{O}$-operators. Given an $\mathcal{O}$-operator $T : M \rightarrow \mathfrak{g}$, a subalgebra $\mathfrak{h} \subset \mathfrak{g}$ and a suitable subspace $E \subset \mathfrak{g}$, we construct under certain conditions, a new $\mathcal{O}$-operator over the Lie algebra $\mathfrak{h} / (E \cap \mathfrak{h})$ (cf. Theorem \ref{mr-main}).

\medskip

In the classical formulation of biHamiltonian mechanics, Poisson structures come up with Nijenhuis tensors suitably compatible with Poisson structures \cite{magri-morosi}. Such structures are called Poisson-Nijenhuis (PN) structures \cite{ks-magri}. It turns out that there is a hierarchy of compatible Poisson structures. These notions and subsequent results has been extended to the context of associative $\mathcal{O}$-operators by introducing $\mathcal{ON}$-structures \cite{sheng}. In Section \ref{sec-comp-o} we first introduce compatible $\mathcal{O}$-operators on modules over Lie algebras and study its relation with associated (pre-)Lie structures. In Section \ref{sec-on} we study Poisson-Nijenhuis structures in the context of $\mathcal{O}$-operators on Lie algebras. 
A Nijenhuis structure on $M$ over $\mathfrak{g}$ consists of a pair $(N, S)$ of linear maps $N \in \mathrm{End}(\mathfrak{g})$ and $S \in \mathrm{End}(M)$ that generates an infinitesimal deformation of the dual $\mathfrak{g}$-module $M^*$ (Definition \ref{defn-nij-str}). We introduce $\mathcal{ON}$-structures on $M$ over $\mathfrak{g}$ as a triple $(T,N,S)$ in which $T$ is an $\mathcal{O}$-operator, $(N, S)$ a Nijenhuis structure on $M$ over $\mathfrak{g}$ satisfying some compatibility relations (Definition \ref{defn-on-str}). We show that for each $k \geq 0$, the linear maps $ T_k := N^k \circ T : M \rightarrow \mathfrak{g}$ are $\mathcal{O}$-operators which are pairwise compatible (Theorem \ref{final-thm-comp-o}). 

\medskip

In the next, we consider strong Maurer-Cartan equation in a twilled Lie algebra (matched pair of Lie algebras). We show that a solution of the strong MC equation in a twilled Lie algebra induces an $\mathcal{ON}$-structure (Theorem \ref{final-thm-strong}), hence, a hierarchy of compatible $\mathcal{O}$-operators (Corollary \ref{cor-strong-mc}). Conversely, we prove that an $\mathcal{ON}$-structure in which the $\mathcal{O}$-operator is invertible induces a solution of the strong Maurer-Cartan equation in a certain twilled Lie algebra (Theorem \ref{thm-strong-last}).

\medskip

In \cite{hitchin} Hitchin introduced a notion of generalized complex structure unifying both symplectic and complex structures. A generalized complex structure has an underlying Poisson structure. Motivated from this, in Section \ref{sec-gcs}, we introduce generalized complex structure on $M$ over the Lie algebra $\mathfrak{g}$ as a linear map $J : \mathfrak{g} \oplus M \rightarrow \mathfrak{g} \oplus M$ of the form
\begin{align*}
J = \left( \begin{array}{cc}
N    &      T\\
\sigma     &    -S
\end{array}
\right),
\end{align*}
satisfying $J^2 = - \mathrm{id}$ and some integrability condition (Definition \ref{defn-gcs-o}). In Theorem \ref{gcs-o-char}, we gave a characterization of a generalized complex structure $J$ in terms of its structure components. 

\medskip

In Section \ref{sec-holo-o}, we introduce holomorphic r-matrices as Lie algebra analog of holomorphic Poisson structures \cite{geng-st-xu}. Finally, using a characterization of holomorphic r-matrices, we end this paper by introducing holomorphic $\mathcal{O}$-operators. Deformations of Lie algebra $\mathcal{O}$-operators are studied in \cite{tang-et} from cohomological perspectives. In a forth coming paper, we aim to study cohomology and deformations of holomorphic $\mathcal{O}$-operators motivated from holomorphic Poisson geometry.

\medskip

All vector spaces and linear maps in this paper are over a field of characteristic $0$ unless otherwise stated.
 
\section{Lie algebras and Nijenhuis tensors}\label{sec-prelim}

Let $\mathfrak{g} = (\mathfrak{g}, [~, ~])$ be a Lie algebra. A $\mathfrak{g}$-module (also called a representation of $\mathfrak{g}$) consists of a vector space $M$ together with a bilinear map (called the action)
\begin{center}
$\bullet : \mathfrak{g} \times M \rightarrow M, ~~~~ (x, m ) \mapsto x \bullet m  \quad \text{ satisfying }$
\end{center}
\begin{center}
$[x, y] \bullet m = x \bullet ( y \bullet m ) - y \bullet (x \bullet m), ~ \text{ for } x, y \in \mathfrak{g}, m \in M.$
\end{center}

Thus it follows that the Lie algebra $\mathfrak{g}$ is a module over itself with the action given by $x \bullet y = [x, y]$, for $x, y \in \mathfrak{g}$. It is called the adjoint representation of $\mathfrak{g}$. The dual vector space $\mathfrak{g}^*$ also carries a $\mathfrak{g}$-module structure with the action given by $\langle x \bullet \alpha, y \rangle = - \langle \alpha, [x, y] \rangle$, for $x, y \in \mathfrak{g}$ and $\alpha \in \mathfrak{g}^*$.

Let $(\mathfrak{g}, [~,~])$ be a Lie algebra and $(M, \bullet)$ be a $\mathfrak{g}$-module. Then the direct sum $\mathfrak{g} \oplus M$ carries a Lie bracket
\begin{align}\label{semi-dir-brkt}
[ (x,m), (y, n)] := ( [x,y] , x \bullet n - y \bullet m),
\end{align}
for $ (x,m), (y, n) \in \mathfrak{g} \oplus M$. This is called the semi-direct product and often denoted by $\mathfrak{g} \ltimes M$.

Let $\mathfrak{g}$ be a Lie algebra and $(M, \bullet)$ be a $\mathfrak{g}$-module. The Chevalley-Eilenberg (CE) cohomology of $\mathfrak{g}$ with coefficients in $M$ is given by the cohomology of the cochain complex $\big(  C^\ast_{\mathrm{CE}} (\mathfrak{g}, M), \delta_{\mathrm{CE}}   \big)$ where $C^n_{\mathrm{CE}} (\mathfrak{g}, M) := \mathrm{Hom} (\wedge^n \mathfrak{g}, M)$, for $n \geq 0$ and $\delta_{\mathrm{CE}} :  C^n_{\mathrm{CE}} (\mathfrak{g}, M) \rightarrow  C^{n+1}_{\mathrm{CE}} (\mathfrak{g}, M)$ given by
\begin{align*}
(\delta_{\mathrm{CE}} f) (x_1, \ldots, x_{n+1}) =~& \sum_{i=1}^{n+1} (-1)^{i+1}~ x_i \bullet f (x_1, \ldots, \widehat{ x_i}, \ldots, x_{n+1}) \\
~&+ \sum_{i < j } (-1)^{i+j} f ( [x_i, x_j], x_1, \ldots, \widehat{ x_i}, \ldots, \widehat{ x_j}, \ldots, x_{n+1} ),
\end{align*}
for $f \in  C^n_{\mathrm{CE}} (\mathfrak{g}, M)$ and $x_1, \ldots, x_{n+1} \in \mathfrak{g}.$

\begin{defn}
Let $(\mathfrak{g}, [~,~])$ be a Lie algebra. A Nijenhuis operator on $\mathfrak{g}$ is a linear map $N :\mathfrak{g} \rightarrow \mathfrak{g}$ satisfying
\begin{align*}
[Nx, Ny] = N ( [Nx, y] + [x, Ny] - N [x,y]), ~ \text{for } x, y \in \mathfrak{g}.
\end{align*}
\end{defn}
If $N$ is a Nijenhuis operator on $\mathfrak{g}$, then the deformed bracket
\begin{align*}
[x,y]_N := [Nx, y] + [x, Ny] - N [x,y]
\end{align*}
is a new Lie bracket on $\mathfrak{g}$ and $N : (\mathfrak{g}, [~, ~]_N) \rightarrow (\mathfrak{g}, [~, ~])$ is a morphism of Lie algebras.

We have more interesting results about Nijenhuis operators \cite{ks-magri}.
\begin{prop} Let $N$ be  a Nijenhuis operator on the Lie algebra $\mathfrak{g}$. Then for all $k, l \in \mathbb{N}$,
\begin{itemize}
\item[(i)] $N^k$ is a Nijenhuis operator on $\mathfrak{g}$, hence, $(\mathfrak{g}, [~, ~]_{N^k})$ is a Lie algebra.
\item[(ii)] $N^l$ is a Nijenhuis operator on the Lie algebra $(\mathfrak{g}, [~,~]_{N^k})$. Moreover, the deformed brackets $([~,~]_{N^k})_{N^l}$ and $[~, ~]_{N^{k+l}}$ coincide, Hence $N^l$ is a Lie algebra morphism from $(\mathfrak{g}, [~, ~]_{N^{k+l}} )$ to $(\mathfrak{g} , [~, ~]_{N^k})$.
\item[(iii)] The Lie brackets $[~,~]_{N^k}$ and $[~,~]_{N^l}$ on $\mathfrak{g}$ are compatible in the sense that any linear combinations of them is also a Lie bracket on $\mathfrak{g}$. 
\end{itemize}
\end{prop}

\section{$\mathcal{O}$-operators}\label{sec-o-op}
In this section, we first recall $\mathcal{O}$-operators and some basic properties of that \cite{bai,bai-guo-ni}. Then we define gauge transformations and reductions of $\mathcal{O}$-operators.

\begin{defn}\label{defn-o-op}
Let $\mathfrak{g}$ be a Lie algebra and $(M, \bullet)$ be a $\mathfrak{g}$-module. An $\mathcal{O}$-operator on $M$ over the Lie algebra $\mathfrak{g}$ is a linear map $T : M \rightarrow \mathfrak{g}$ satisfying
\begin{align*}
[T(m), T(n)] = T (T(m) \bullet n - T(n) \bullet m), ~ \text{for } m, n \in M.
\end{align*}
\end{defn}

Let $T$ be an $\mathcal{O}$-operator on $M$ over $\mathfrak{g}$. Then $M$ carries a Lie algebra structure with bracket
\begin{align}\label{new-brkt}
[m,n]^T := T(m) \bullet n - T(n) \bullet m, ~ \text{for } m, n \in M.
\end{align}
We denote this Lie algebra by $M^T$. Moreover, $\mathrm{ker}(T) \subset M^T$ is a subalgebra, called the isotropy subalgebra. This is in fact an ideal. The image of $T$, $\mathrm{im}(T) \subset \mathfrak{g}$ is also a subalgebra.

\begin{prop}\label{o-char}
A linear map $T: M \rightarrow \mathfrak{g}$ is an $\mathcal{O}$-operator on $M$ over $\mathfrak{g}$ if and only if the graph
\begin{align*}
\mathrm{Gr} (T) := \{ (T(m), m) | m \in M \} \subset \mathfrak{g} \oplus M
\end{align*}
is a subalgebra of the semi-direct product $\mathfrak{g} \ltimes M.$
\end{prop}


Let $\mathfrak{g} = (\mathfrak{g} , [~, ~])$ be a Lie algebra. Then one can extend the Lie bracket on $\mathfrak{g}$ to the full exterior algebra $\wedge^\ast \mathfrak{g} = \oplus_{n \geq 0} \wedge^n \mathfrak{g}$ by the following rules
\begin{align*}
&[P, Q] = -(-1)^{(p-1)(q-1)}[Q, P],\\
&[P, Q \wedge R ] = [P, Q] \wedge R + (-1)^{(p-1)q} Q \wedge [P, R], \text{ for } P \in \wedge^p \mathfrak{g}, ~ Q \in \wedge^q \mathfrak{g} \text{ and } R \in \wedge^r \mathfrak{g}.
\end{align*}

\begin{defn}
An element ${\bf r} \in \wedge^2 \mathfrak{g}$ is called a classical r-matrix (or a solution of the classical Yang-Baxter equation) if ${\bf r}$ satisfies $[{\bf r}, {\bf r}] = 0$. 
\end{defn}


Classical r-matrices are Lie algebra analog of Poisson structures \cite{kuper}. There is a close connection between classical r-matrices and $\mathcal{O}$-operators.

\begin{lemma}\label{lemma-r-mat}
An element ${\bf r} \in \wedge^2 \mathfrak{g}$ is a classical r-matrix if and only if the induced map ${\bf r}^\sharp : \mathfrak{g}^* \rightarrow \mathfrak{g}, ~ \alpha \mapsto {\bf r} (\alpha, ~)$ is an $\mathcal{O}$-operator on the coadjoint representation $\mathfrak{g^*}$ over the Lie algebra $\mathfrak{g}$.
\end{lemma}






\subsection{Gauge transformations}
Gauge transformations of Poisson structures by suitable closed $2$-forms was defined by \v{S}evera and Weinstein \cite{sev-wein}. Since $\mathcal{O}$-operators are generalization of Poisson structures, we may define gauge transformations of $\mathcal{O}$-operators. We proceed as follows.

Let $\mathfrak{g}$ be a Lie algebra and $M$ be a $\mathfrak{g}$-module. Let $L \subset \mathfrak{g} \ltimes M$ be a Lie subalgebra of the semi-direct product. For any linear map $B : \mathfrak{g} \rightarrow M$, we define a subspace
\begin{align*}
\tau_B (L ) := \{ (x, m + B(x)) |~(x, m) \in L \} \subset \mathfrak{g} \oplus M.
\end{align*}

\begin{prop}
The subspace $\tau_B (L) \subset \mathfrak{g} \oplus M$ is a Lie subalgebra of the semi-direct product $\mathfrak{g} \ltimes M$ if and only if $B$ is a $1$-cocycle in the cohomology of the Lie algebra $\mathfrak{g}$ with coefficients in $M$.
\end{prop}

\begin{proof}
For any $(x,m), (y, n) \in L$, we have
\begin{align*}
[ (x, m + B(x)) , (y, n + B(y)) ]
&= ( [x, y], x \bullet (n + B(y)) - y \bullet (m + B(x)) ) \\
&= ([x, y], x \bullet n - y \bullet m + x \bullet B(y) - y \bullet B(x)).
\end{align*}
It is in $\tau_B (L)$ if and only if $x \bullet B(y) - y \bullet B(x) = B ([x, y])$, or, equivalently, $B$ is a $1$-cocycle in the cohomology of $\mathfrak{g}$ with coefficients in $M$.
\end{proof}

Let $T: M \rightarrow \mathfrak{g}$ be an $\mathcal{O}$-operator on $M$ over the Lie algebra $\mathfrak{g}$. Consider the graph $\mathrm{Gr} (T) := \{ (T(m), m)|~ m \in M \} \subset \mathfrak{g} \ltimes M$ which is a Lie subalgebra of the semi-direct product. For any $1$-cocycle $B : \mathfrak{g} \rightarrow M$, we consider the deformed subalgebra $\tau_B ( \mathrm{Gr}(T)) \subset \mathfrak{g} \ltimes M$. The question is whether this subalgebra is the graph of a linear map from $M $ to $\mathfrak{g}$ ?

If the linear map $\mathrm{id}_M + B \circ T : M \rightarrow M$ is invertible, then $\tau_B (\mathrm{Gr}(T))$ is the graph of the linear map $T \circ (\mathrm{id}_M + B \circ T)^{-1} : M \rightarrow \mathfrak{g}$. In such a case, the $1$-cocycle $B$ is called $T$-admissible. Therefore, by Proposition \ref{o-char}, the linear map  $T \circ (\mathrm{id}_M + B \circ T)^{-1} : M \rightarrow \mathfrak{g}$ is an $\mathcal{O}$-operator on $M$ over the Lie algebra $\mathfrak{g}$. This $\mathcal{O}$-operator is called the gauge transformation of $T$ associated with $B$, and denoted by $T_B$.

\begin{remark}
(i) $\mathrm{im} (T) = \mathrm{im} (T_B)$.\\ (ii) If $T$ is invertible then $T_B$ is so, and
\begin{align*}
T_B^{-1} = (\mathrm{id}_M + B \circ T) \circ T^{-1} = T^{-1} + B.
\end{align*}
\end{remark}

\begin{prop}\label{iso-m-mt}
The Lie algebra structures on $M$ induced from $\mathcal{O}$-operators $T$ and $T_B$ are isomorphic.
\end{prop}

\begin{proof}
Consider the invertible linear map $\mathrm{id}_M + B \circ T : M \rightarrow M$. Then for any $m, n \in M$, we have
\begin{align*}
&[ (\mathrm{id}_M + B \circ T) (m), (\mathrm{id}_M + B \circ T)(n)]^{T_B} \\
&= T_B (\mathrm{id}_M + B \circ T) (m) \bullet (\mathrm{id}_M + B \circ T) (n) - T_B (\mathrm{id}_M + B \circ T)(n) \bullet (\mathrm{id}_M + B \circ T)(m) \\
&= T(m) \bullet n + T(m) \bullet BT(n) - T(n) \bullet m - T(n) \bullet BT(m) \\
&= T(m) \bullet n - T(n) \bullet m + B ( [T(m), T(n)]) \\
&= [m, n]^T +B \circ T ([m, n]^T) = (\mathrm{id}_M + B \circ T)([m, n]^T).
\end{align*}
Hence the proof.
\end{proof}

\begin{remark}
Gauge transformations of $\mathcal{O}$-operators (generalized Rota-Baxter operators of Uchino \cite{uchino}) on bimodules over associative algebras can be defined in a similar manner. Moreover, these two constructions of gauge transformations are related by the standard skew-symmetrization from associative algebras to Lie algebras. 
\end{remark}

\subsection{Reductions}\label{subsec-red}

In this subsection, we extend the well-known Marsden-Ratiu Poisson reduction theorem \cite{mars-ratiu} to $\mathcal{O}$-operators. In the classical case, this reduction theorem allows one, under certain conditions, to construct a new Poisson structure on the quotient $N / \mathcal{F}$, $N$ being a submanifold of a Poisson manifold $M$ and $\mathcal{F}$ the foliation associated with an integrable distribution $E \cap TN$ with $E$ a vector subbundle of $TM$ restricted to $N$.

Let $\mathfrak{g}$ be a Lie algebra and $M$ be a $\mathfrak{g}$-module. Let $T : M \rightarrow \mathfrak{g}$ be an $\mathcal{O}$-operator on $M$ over the Lie algebra $\mathfrak{g}$. Suppose $\mathfrak{h} \subset \mathfrak{g}$ is a Lie subalgebra, $E \subset \mathfrak{g}$ a subspace satisfying the property that
 the quotient $\mathfrak{h}/ E \cap \mathfrak{h}$ is a Lie algebra and the projection $\pi: \mathfrak{h} \rightarrow \mathfrak{h} / E \cap \mathfrak{h}$ is a morphism of Lie algebras.

Let $N \subset M$ be an $\mathfrak{h}$-module. Define a subspace
\begin{align*}
(E \cap \mathfrak{h})^0_N := \{ n \in N |~ x \bullet n = 0, ~ \forall x \in E \cap \mathfrak{h} \} \subset N. 
\end{align*}
Then $(E \cap \mathfrak{h})^0_N$ is a $\mathfrak{h}/ E \cap \mathfrak{h}$-module with the action given by
~$| h | \bullet n = h \bullet n.$

\begin{defn}
Let $\mathfrak{g}$ be a Lie algebra and $T : M \rightarrow \mathfrak{g}$ be an $\mathcal{O}$-operator on a $\mathfrak{g}$-module $M$. A triple $(\mathfrak{h}, E, N)$ as above is said to be reducible if there is an $\mathcal{O}$-operator $\overline{T} : (E \cap \mathfrak{h})^0_N \rightarrow \mathfrak{h}/ E \cap \mathfrak{h} $ such that for any $m, n \in (E \cap \mathfrak{h})^0_N$, we have $\overline{T} (m) \bullet n = T (m) \bullet n$.
\end{defn}

The Marsden-Ratiu reduction theorem for $\mathcal{O}$-operators can be stated as follows.

\begin{thm}\label{mr-main}
Let $\mathfrak{g}$ be a Lie algebra and $M$ be a $\mathfrak{g}$-module. Let $T: M \rightarrow \mathfrak{g}$ be an $\mathcal{O}$-operator on $M$ over the Lie algebra $\mathfrak{g}$. 
If $T ((E \cap \mathfrak{h})^0_N) \subset \mathfrak{h}$ then $(\mathfrak{h}, E, N)$ is reducible.
\end{thm}

\begin{proof}
For any $m, n \in (E \cap \mathfrak{h})^0_N$, we claim that $T(m) \bullet n \in (E \cap \mathfrak{h})^0_N$. This follows as for any $x \in E \cap \mathfrak{h}$, we have
\begin{align*}
x \bullet ( T(m) \bullet n ) = [x, T(m) ] \bullet n  - T (m) \bullet ( x \bullet n).
\end{align*}
First observe that $\pi [ x, T(m)] = [ \underbrace{\pi (x)}_{= 0}, \pi T(m) ] = 0$. Hence $[x, T(m)] \in E \cap \mathfrak{h}$. Therefore, the first term of the right hand side vanishes as $[x, T(m)] \in E \cap \mathfrak{h}$ and $n \in (E \cap \mathfrak{h})_N^0$. The second term of the right hand side vanishes as $x \in E \cap \mathfrak{h}$ and $n \in (E \cap \mathfrak{h})^0_N$. Therefore, we get $T(m ) \bullet n \in (E \cap \mathfrak{h})^0_N$. We define $\overline{T} : (E \cap \mathfrak{h})^0_N \rightarrow \mathfrak{h}/ E \cap \mathfrak{h}$ by $\overline{T}(m) := |T(m)|$ the class of $T(m)$. Then we have
\begin{align*}
[ \overline{T}(m), \overline{T}(n)] = [ |T(m)|, |T(n)| ] = | [T(m), T(n)] |.
\end{align*}
On the other hand
\begin{align*}
\overline{T} ( \overline{T}(m) \bullet n - \overline{T} (n) \bullet m ) = \overline{T} ( |T(m)| \bullet n - |T(n)| \bullet m) =~&  \overline{T} ( T(m) \bullet n - T(n) \bullet m) \\
=~& | T ( T(m) \bullet n - T(n) \bullet m) |  = | [T(m), T(n)] |.
\end{align*}
Hence $\overline{T}$ is an $\mathcal{O}$-operator on $(E \cap \mathfrak{h})^0_N$ over the Lie algebra $\mathfrak{h} / E \cap \mathfrak{h}$. Moreover $\overline{T}(m) \bullet n = | T(m)| \bullet n = T(m) \bullet n$. Hence the triple $(\mathfrak{h}, E, N)$ is reducible.
\end{proof}

As consequences, we obtain the followings.

\medskip

(i) Let $T : M \rightarrow \mathfrak{g}$ be an $\mathcal{O}$-operator and $\mathfrak{h} \subset \mathfrak{g}$ be a Lie subalgebra. If $N \subset M$ is an $\mathfrak{h}$-submodule and $T(N ) \subset \mathfrak{h}$, then the restriction $T : N \rightarrow \mathfrak{h}$ is an $\mathcal{O}$-operator on $N$ over $\mathfrak{h}$.

\medskip

(ii) Let $T : M \rightarrow \mathfrak{g}$ be an $\mathcal{O}$-operator and $E \subset \mathfrak{g}$ be an ideal. Then $E^0_M$ is an $\mathfrak{g}/ E$-module and the map $\overline{T}: E^0_M \rightarrow \mathfrak{g}/ E,~ m \mapsto |T(m)|$ is an $\mathcal{O}$-operator on $E^0_M$ over $\mathfrak{g}/ E$ satisfying $\overline{T}(m) \bullet n = T(m) \bullet n$, for $m, n \in E^0_M$.

\section{Compatible $\mathcal{O}$-operators}\label{sec-comp-o}

\begin{defn}
Two $\mathcal{O}$-operators $T_1, T_2 : M \rightarrow \mathfrak{g}$ on $M$ over the Lie algebra $\mathfrak{g}$ are said to be compatible if their sum $T_1 + T_2 : M \rightarrow \mathfrak{g}$ is also an $\mathcal{O}$-operator.
\end{defn}

Note that the condition in the above definition is equivalent to
\begin{align}\label{comp-o-equiv}
[T_1 (m), T_2 (n)] + [ T_2 (m), T_1 (n)] = T_1 (T_2 (m) \bullet n - T_2 (n) \bullet m) + T_2 ( T_1 (m) \bullet n - T_1 (n) \bullet m).
\end{align}
This also implies that for any $\mu, \lambda \in \mathbb{K}$, the linear combination $\mu T_1 + \lambda T_2$ is an $\mathcal{O}$-operator.

If two Poisson structures are compatible and one of them is non-degenerate (i.e. obtained from a symplectic structure) then one can construct a Nijenhuis tensor on the manifold \cite{vaisman}. Since $\mathcal{O}$-operators are generalization of Poisson structures, one can extend this result in our result.

\begin{prop}\label{prop-inver-nij}
Let $T_1, T_2 : M \rightarrow \mathfrak{g}$ be two $\mathcal{O}$-operators on $M$ over the Lie algebra $\mathfrak{g}$. If $T_1, T_2$ are compatible and $T_2$ is invertible then $N = T_1 \circ T_2^{-1} : \mathfrak{g} \rightarrow \mathfrak{g}$ is a Nijenhuis operator on the Lie algebra $\mathfrak{g}$.
Conversely, if $T_1, T_2$ are both invertible and $N$ is a Nijenhuis tensor then $T_1 , T_2$ are compatible.
\end{prop}

\begin{proof}
Let $T_1, T_2$ be compatible and $T_2$ invertible. For any $x, y \in \mathfrak{g}$, there exists (unique) elements $m, n \in M$ such that $T_2(m) = x$ and $T_2 (n) = y$. Then
\begin{align*}
&[Nx, Ny] - N ([Nx, y] + [x, Ny]) + N^2 [x, y] \\
&= [NT_2 (m), NT_2 (n)] - N ( [NT_2(m), T_2 (n)] + [T_2 (m), NT_2 (n)] ) + N^2 [T_2 (m), T_2 (n)]  \\
&= [T_1 (m), T_1 (n)] - N ( [T_1(m), T_2 (n)] + [T_2 (m), T_1 (n)] ) + N^2 [T_2 (m), T_2 (n)]  \\
&= T_1 ( T_1(m) \bullet n - T_1(n) \bullet m ) - NT_1 ( T_2(m) \bullet n - T_2 (n) \bullet m ) - NT_2 (  T_1(m) \bullet n - T_1(n) \bullet m ) \\& ~~~
+ N^2 T_2 (T_2(m) \bullet n - T_2 (n) \bullet m) ~~~(\text{as } T_1, T_2 \text{ are } \mathcal{O}\text{-operators and by } (\ref{comp-o-equiv}))\\
& = 0.
\end{align*}

Conversely, if $N$ is a Nijenhuis tensor then for all $m, n \in M$,
\begin{align*}
[NT_2 (m), NT_2 (n)] = N ( [NT_2(m), T_2(n)] + [T_2(m), NT_2(n)]) - N^2 [ T_2(m), T_2(n)].
\end{align*}
This implies that 
\begin{align*}
T_1 ( T_1(m) \bullet n - T_1(n) \bullet m) = N ( [T_1(m), T_2(n)] + [T_2(m), T_1(n)]) - NT_1 ( T_2(m) \bullet n - T_2(n) \bullet m).
\end{align*}
Since $N$ is invertible, we may apply $N^{-1}$ to both sides to get the identity (\ref{comp-o-equiv}). Hence $T_1$ and $T_2$ are compatible.
\end{proof}

\subsection{Compatible pre-Lie algebras}
In this subsection, we recall pre-Lie algebras and their relation with $\mathcal{O}$-operators. We show that compatible $\mathcal{O}$-operators give rise to compatible pre-Lie algebras.

\begin{defn}
A (left) pre-Lie algebra is a vector space $L$ together with a linear map $\square : L \otimes L \rightarrow L$ satisfying
\begin{align*}
(x ~\square~ y) ~\square~ z - x ~\square~ ( y ~\square~ z) = (y ~\square~ x) ~\square~ z - y ~\square~ ( x ~\square~ z), ~ \text{ for } x, y, z \in L.
\end{align*}
In this case, $\square$ is called a pre-Lie product on $L$.
\end{defn}

The connection between $\mathcal{O}$-operators and pre-Lie algebras is given by the following \cite{bai}.
\begin{prop}
Let $T : M \rightarrow \mathfrak{g}$ be an $\mathcal{O}$-operator on $M$ over the Lie algebra $\mathfrak{g}$. Then the product $\square_T : M \otimes M \rightarrow M, ~ m ~\square_T~ n = T(m) \bullet n$ is a pre-Lie product on $M$.
\end{prop}

\begin{defn}
Two pre-Lie products $\square_1$ and $\square_2$ on a vector space $L$ are said to compatible if for all $\mu, \lambda \in \mathbb{K}$, the sum $\mu \square_1 + \lambda \square_2$ is also a pre-Lie product on $L$. 
\end{defn}

This is equivalent to 
\begin{align*}
&(x ~\square_1~ y) ~\square_2~ z - x ~\square_1~ ( y ~\square_2~ z) + (x ~\square_2~ y) ~\square_1~ z - x ~\square_2~ ( y ~\square_1~ z)\\ &= (y ~\square_1~ x) ~\square_2~ z - y ~\square_1~ ( x ~\square_2~ z) + (y ~\square_2~ x) ~\square_1~ z - y ~\square_2~ ( x ~\square_1~ z).
\end{align*}
\begin{prop}
Let $T_1, T_2 : M \rightarrow \mathfrak{g}$ be two compatible $\mathcal{O}$-operators on $M$ over the Lie algebra $\mathfrak{g}$. Then the pre-Lie products $\square_{T_1}$ and $\square_{T_2}$ on $M$ are compatible.
\end{prop}

\section{$\mathcal{ON}$-structures}\label{sec-on}
In this section, we study Nijenhuis structure on a module over a Lie algebra. Then we introduce $\mathcal{ON}$-structures and show that an $\mathcal{ON}$-structure induces a hierarchy of compatible $\mathcal{O}$-operators.

\subsection{Nijenhuis structures on modules over Lie algebras}
Let $\mathfrak{g}$ be a Lie algebra and $M$ be a $\mathfrak{g}$-module. An infinitesimal deformation of the $\mathfrak{g}$-module $M$ is given by sums
\begin{align*}
[x, y]_t = [x, y] + t [x, y]_1 ~~~ \text{ and } ~~~ x \bullet_t m = x \bullet m + t~ x \bullet_1 m, ~~~\text{ for } x, y \in \mathfrak{g}, m \in M,
\end{align*}
where $[~, ~]_1$ is a skew-symmetric bracket on $\mathfrak{g}$ and $\bullet_1 : \mathfrak{g} \times M \rightarrow M$ is a bilinear  map such that $(\mathfrak{g}, [~,~]_t)$ is a Lie algebra and $\bullet_t$ defines a $(\mathfrak{g}, [~,~]_t)$-module on $M$. Thus it follows that the following identities are hold: for $x, y, z \in \mathfrak{g}$ and $m \in M$,
\begin{center}
$[ x, [y, z]_t ]_t + [y, [z, x]_t ]_t + [z, [x, y]_t ]_t =0,$
\end{center}
\begin{center}
$[x, y]_t \bullet_t m = x \bullet_t ( y \bullet_t m ) - y \bullet_t ( x \bullet_t m).$
\end{center}
These two conditions are equivalent to the following identities
\begin{align}
&[ x, [y, z]_1 ] + [y, [z, x]_1 ] + [z, [x, y]_1 ] + [ x, [y, z] ]_1 + [y, [z, x] ]_1 + [z, [x, y] ]_1 = 0, \label{cond-11}\\
&[ x, [y, z]_1 ]_1 + [y, [z, x]_1 ]_1 + [z, [x, y]_1 ]_1 = 0, \label{cond-22}\\
&[x, y]_1 \bullet_1 m = x \bullet_1 ( y \bullet_1 m ) - y \bullet_1 ( x \bullet_1 m), \label{cond-33}\\
&[x, y] \bullet_1 m + [x, y]_1 \bullet m = x \bullet ( y \bullet_1 m ) - y \bullet ( x \bullet_1 m) + x \bullet_1 ( y \bullet m ) - y \bullet_1 ( x \bullet m). \label{cond-44}
\end{align}
The condition (\ref{cond-11}) implies that $[~,~]_1$ is a $2$-cocycle of the Lie algebra $\mathfrak{g}$ with coefficients in itself. The condition (\ref{cond-22}) says that $[~,~]_1$ is a Lie bracket on $\mathfrak{g}$ and (\ref{cond-33}) says that $\bullet_1$ defines a $(\mathfrak{g} , [~,~]_1)$-module structure on $M$. Finally (\ref{cond-44}) is equivalent to the fact $(M, \bullet + \bullet_1)$ is a module for the Lie algebra $(\mathfrak{g}, [~, ~]_t)$ for $t=1$.

\begin{defn}
Let $([~,~]_t, \bullet_t)$ and $([~,~]_t', \bullet_t')$ be two infinitesimal deformations of a $\mathfrak{g}$-module $M$. They are said to be equivalent if there exist linear maps $N \in \mathrm{End} (\mathfrak{g})$ and $S \in \mathrm{End} (M)$ such that $(\mathrm{id}_\mathfrak{g} + tN , \mathrm{id}_M + tS)$ is a homomorphism from the $(\mathfrak{g}, [~,~]_t')$-module $(M, \bullet_t') $ to the $(\mathfrak{g}, [~,~]_t)$-module $(M, \bullet_t) $, i.e. the followings hold
\begin{align*}
(\mathrm{id}_\mathfrak{g} + tN) [x, y]_t' =~& [ (\mathrm{id}_\mathfrak{g} + tN)(x), (\mathrm{id}_\mathfrak{g} + tN)(y)]_t, \\
(\mathrm{id}_M + tS) ( x \bullet_t' m ) =~& (\mathrm{id}_\mathfrak{g} + tN) (x) \bullet_t (\mathrm{id}_M + tS)(m).
\end{align*}
\end{defn}

An infinitesimal deformation $([~,~]_t, \bullet_t)$ of the $\mathfrak{g}$-module $M$ is said to be trivial if it is equivalent to the undeformed one $([~,~]_t' = [~,~], \bullet_t' = \bullet)$.
Thus an infinitesimal deformation $([~,~]_t, \bullet_t)$ is trivial if and only if there exists $N \in \mathrm{End} (\mathfrak{g})$ and $S \in \mathrm{End} (M)$ satisfying
\begin{align}
[x, y]_1 =~& [Nx, y] + [x, Ny] - N[x, y], \label{iden-111}\\
N[x, y]_1 =~& [Nx, Ny], \label{iden-222}\\
x \bullet_1 m =~& Nx \bullet m + x \bullet Sm - S ( x \bullet m), \label{iden-333}\\
S ( x \bullet_1 m) = ~& N(x) \bullet S(m),~ \text{ for } x, y \in \mathfrak{g} \text{ and } m \in M. \label{iden-444}
\end{align}
It follows from (\ref{iden-111}) and (\ref{iden-222}) that $N$ is a Nijenhuis tensor for the Lie algebra $\mathfrak{g}$. Similarly, from (\ref{iden-333}) and (\ref{iden-444}), we get that
\begin{align}\label{id-t}
N(x) \bullet S(m) = S (Nx \bullet m + x \bullet Sm - S ( x \bullet m)).
\end{align}
Thus, in a trivial infinitesimal deformation, $N$ is a Nijenhuis tensor for the Lie algebra $\mathfrak{g}$ and satisfying the identity (\ref{id-t}). In fact, any such operators $N, S$ generate a trivial infinitesimal deformation of the $\mathfrak{g}$-module $M$.

\begin{thm}
Let $\mathfrak{g}$ be a Lie algebra and $M$ be a $\mathfrak{g}$-module. Let $N \in \mathrm{End} (\mathfrak{g})$ be a Nijenhuis operator on $\mathfrak{g}$ and $S \in \mathrm{End} (M)$ satisfies the condition (\ref{id-t}). Then $([~,~]_t, \bullet_t)$ is a trivial infinitesimal deformation of the $\mathfrak{g}$-module $M$ where 
\begin{align*}
[x, y]_t =~& [x, y]+ t ([Nx, y] + [x, Ny] - N [x, y])~~  \text{ and } \\
x \bullet_t m =~& x \bullet m + t ( Nx \bullet m + x \bullet Sm - S ( x \bullet m)), \text{ for } x, y \in \mathfrak{g}, m \in M.
\end{align*}
\end{thm}
\begin{proof}
It is a routine calculation to verify that the identities (\ref{cond-11})-(\ref{cond-44}) holds. Hence $([~,~]_t, \bullet_t)$ is a deformation of the $\mathfrak{g}$-module $M$. Finally, the conditions (\ref{iden-111})-(\ref{iden-444}) of the triviallity of a deformation suggests that $([~,~]_t, \bullet_t)$ is trivial.
\end{proof}

Note that the conditions that $N$ is a Nijenhuis tensor and $S$ satisfies the identity (\ref{id-t}) can be expressed simply by the following result.

\begin{prop}
Let $\mathfrak{g}$ be a Lie algebra and $M$ be a $\mathfrak{g}$-module. A linear map $N \in \mathrm{End} (\mathfrak{g})$ is a Nijenhuis tensor on $\mathfrak{g}$ and a linear map $S \in \mathrm{End} (M)$ satisfies the identity (\ref{id-t}) if and only if $N \oplus S : \mathfrak{g} \oplus M \rightarrow \mathfrak{g} \oplus M$ is a Nijenhuis operator  on the semi-direct product Lie algebra $\mathfrak{g} \ltimes M$.
\end{prop}

\begin{defn}\label{defn-nij-str}
Let $\mathfrak{g}$ be a Lie algebra and $M$ be a $\mathfrak{g}$-module. A pair $(N, S)$ consisting of linear maps $N \in \mathrm{End} (\mathfrak{g})$ and $S \in \mathrm{End}(M)$ is called a Nijenhuis structure on $M$ if $N$ and $S^*$ generate a trivial infinitesimal deformation of the dual $\mathfrak{g}$-module $M^*$. 
\end{defn}

Note that the condition of the above definition is equivalent to the fact that $N$ is a Nijenhuis tensor on $\mathfrak{g}$ and 
\begin{align}\label{nij-second}
N(x) \bullet S(m) = S ( N(x) \bullet m) + x \bullet S^2(m) - S ( x \bullet S(m)), ~ \text{for } x \in \mathfrak{g}, m \in M.
\end{align}

Let $N : \mathfrak{g} \rightarrow \mathfrak{g}$ be a Nijenhuis operator on the Lie algebra $\mathfrak{g}$. Then $(N, N^*)$ is a Nijenhuis structure on the coadjoint module $\mathfrak{g}^*$.

\begin{prop}
Let $(N, S)$ be a Nijenhuis structure on a $\mathfrak{g}$-module $M$. Then the pairs $(N^i, S^i)$ are Nijenhuis structures on the $\mathfrak{g}$-module $M$, for all $i \in \mathbb{N}$.
\end{prop}

Let $(N, S)$ be a Nijenhuis structure on a $\mathfrak{g}$-module $M$. Consider the deformed Lie algebra $(\mathfrak{g}, [~, ~]_N)$. We define a map $\widetilde{\bullet} : \mathfrak{g} \times M \rightarrow M$ by
\begin{align*}
x ~\widetilde{ \bullet }~ m = N(x) \bullet m - x \bullet S(m) + S ( x \bullet m).
\end{align*}

\begin{prop}
The map $\widetilde{\bullet} : \mathfrak{g} \times M \rightarrow M$ defines a representation of the Lie algebra $(\mathfrak{g}, [~,~]_N)$ on $M$.
\end{prop}

\begin{proof}
Since $(N, S)$ is a Nijenhuis structure on the $\mathfrak{g}$-module $M$, the sum $N \oplus S^* : \mathfrak{g} \oplus M^* \rightarrow \mathfrak{g} \oplus M^*$ is a Nijenhuis tensor on the semi-direct product Lie algebra $\mathfrak{g} \ltimes M^*.$ The deformed bracket is given by
\begin{align*}
&[(x, \alpha ), (y, \beta )]_{N \oplus S^*} \\
&=  [ (N \oplus S^*) (x, \alpha), (y, \beta)] + [(x, \alpha), (N \oplus S^*) (y, \beta)] - (N \oplus S^*)[ (x, \alpha), (y, \beta)] \\
&= ([Nx,y] +[x, Ny] - N[x,y], N(x) \bullet \beta - y \bullet S^*(\alpha) + x \bullet S^*(\beta) - N(y) \bullet \alpha - S^* (x \bullet \beta) + S^* ( y \bullet \alpha )) \\
&= ([x,y]_N, x ~\widetilde{ \bullet }~ \beta  - y ~\widetilde{ \bullet }~ \alpha ).
\end{align*}
This shows that $(M^*, ~\widetilde{ \bullet })$ is a module over the Lie algebra $(\mathfrak{g}, [~, ~]_N)$. Hence the dual  $(M, ~\widetilde{ \bullet })$ is also a module over $(\mathfrak{g}, [~, ~]_N).$
\end{proof}

Note that, we may define a bracket $[~, ~]_\sim^T : M \times M \rightarrow M$ by using the representation given in the above proposition
\begin{align*}
[m,n]^T_\sim := T(m) ~\widetilde{ \bullet }~ n - T(n) ~\widetilde{ \bullet }~ m.
\end{align*}

\subsection{$\mathcal{ON}$-structures}

Let $T: M \rightarrow \mathfrak{g}$ be an $\mathcal{O}$-operator on $M$ over the Lie algebra $\mathfrak{g}$. Consider the Lie algebra structure on $M$ with the bracket $[~,~]^T$ given in (\ref{new-brkt}).
Next, let $(N, S)$ be a Nijenhuis structure on $M$ over the Lie algebra $\mathfrak{g}$. Then one can deform the bracket $[~, ~]^T$ by the linear map $S \in \mathrm{End}(M)$ and obtain a new bracket
\begin{align*}
[m, n]^T_S = [ S(m), n ]^T + [m, S(n)]^T - S ([m,n]^T), ~\text{for } m, n \in M.
\end{align*}

\begin{defn}\label{defn-on-str}
Let $T : M \rightarrow \mathfrak{g}$ be an $\mathcal{O}$-operator and $(N, S)$ be a Nijenhuis structure on $M$ over the Lie algebra $\mathfrak{g}$. The triple $(T, N, S)$ is said to be an $\mathcal{ON}$-structure on $M$ over the Lie algebra $\mathfrak{g}$ if the following conditions hold:
\begin{itemize}
\item[$\triangleright$] $N \circ T = T \circ S$,
\item[$\triangleright$] $[m, n ]^{N \circ T} = [m, n ]^T_S$, for $m, n \in M$.
\end{itemize}
\end{defn}

Here the bracket $[~, ~]^{N \circ T}$ is defined similar to (\ref{new-brkt}) where $T$ is replaced by $N \circ T$. If $(T, N, S)$ is an $\mathcal{ON}$-structure, then by the first condition of the above definition, we have
\begin{align*}
[m,n]^T_S + [m,n]^T_\sim = 2 [m,n]^{NT}.
\end{align*}
Hence by the second condition of the above definition, we get $[m,n]^T_\sim = [m,n]^T_S.$

\begin{thm}\label{nt}
Let $(T, N, S)$ be an $\mathcal{ON}$-structure on $M$ over the Lie algebra $\mathfrak{g}$. Then
\begin{itemize}
\item[(i)] $T$ is an $\mathcal{O}$-operator on $(M, \widetilde{\bullet})$ over the deformed Lie algebra $(\mathfrak{g}, [~, ~]_N)$.
\item[(ii)] $N \circ T$ is an $\mathcal{O}$-operator on $M$ over the Lie algebra $\mathfrak{g}$.
\end{itemize}
\end{thm}

\begin{proof}
(i) For any $m,n \in M$, we have
\begin{align}
T( [m, n]_\sim^T) = T ([m,n]^T_S) =~& T ( [Sm, n]^T + [m, Sn]^T - S [m,n]^T ) \nonumber \\
=~& [ TS(m), T(n)] + [ T(m), TS(n)] - TS [m,n]^T \nonumber \\
=~& [ NT(m), T(n)] + [T(m), NT(n)] - N [T(m), T(n)] ~~(\text{as } TS = NT) \nonumber\\
=~& [T(m), T(n)]_N. \label{random}
\end{align}

(ii) By (\ref{random}) and the fact that $N$ is a Nijenhuis tensor, we have
\begin{align*}
NT ([m,n]^{NT}) = NT ([m,n]^T_S) = N ( [Tm, Tn]_N ) = [NT(m), NT(n)].
\end{align*}
Hence $N \circ T$ is an $\mathcal{O}$-operator on $M$ over the Lie algebra $\mathfrak{g}$.
\end{proof}

In a Poisson-Nijenhuis manifold $(M, \pi, N)$, it is known that the Poisson structures $\pi$ and $N \pi$ are compatible \cite{ks-magri}. Here we prove an $\mathcal{O}$-operator version of the above result.

\begin{prop}\label{nt-com}
Let $(T, N, S)$ be an $\mathcal{ON}$-structure on $M$ over the Lie algebra $\mathfrak{g}$. Then $T$ and $N \circ T$ are compatible $\mathcal{O}$-operators.
\end{prop}

\begin{proof}
For any $m, n \in M$, we have
\begin{align*}
[m, n]^{ T + N \circ T } = [m,n ]^T + [m,n]^{ N \circ T } = [m,n]^T + [m,n]^T_S.
\end{align*}
Hence
\begin{align*}
&~(T + N \circ T )( [m, n]^{ T + N \circ T } )\\ &= T ([m,n]^T) + T ([m,n]^T_S )  +  (N \circ T) ([m,n]^T) + (N \circ T) ([m,n]^T_S ) \\
&= T ([m,n]^T) + T ( [Sm, n]^T + [m, Sn]^T - S [m,n]^T) + (N \circ T)([m,n]^T) + (N \circ T) ([m,n]^{N \circ T}) \\
&= T ([m,n]^T) + T ( [Sm, n]^T + [m, Sn]^T) + (N \circ T) ([m,n]^{N \circ T}) \\
&= [Tm, Tn] + ( [TS(m), T(n)] + [T(m), TS(n)] ) + [NT(m), NT(n)] \\ & \qquad \qquad \qquad~~(\text{as } T \text{ and } N \circ T \text{ are } \mathcal{O}\text{-operators})\\
&= [ (T + NT)(m), (T + NT)(n)] ~~~(\text{as } TS = NT). 
\end{align*}
This shows that the sum $T + N \circ T$ is an $\mathcal{O}$-operator on $M$ over $\mathfrak{g}$. Hence the proof.
\end{proof}

In the next proposition, we construct an $\mathcal{ON}$-structure from compatible $\mathcal{O}$-operators.

\begin{prop}
Let $T_1, T_2 : M \rightarrow \mathfrak{g}$ be two compatible $\mathcal{O}$-operators on $M$ over the Lie algebra $\mathfrak{g}$. If $T_2$ is invertible then $(T_2, N = T_1 \circ T_2^{-1}, S = T_2^{-1} \circ T_1)$ is an $\mathcal{ON}$-structure, for $i=1, 2$.
\end{prop}

\begin{proof}
We know from Proposition \ref{prop-inver-nij} that $N = T_1 \circ T_2^{-1}$ is a Nijenhuis tensor on the Lie algebra $\mathfrak{g}$. We will now prove that $S = T_2^{-1} \circ T_1$ satisfies (\ref{nij-second}) to make the pair $(N, S)$ a Nijenhuis structure on $M$ over the Lie algebra $\mathfrak{g}$.

Since $T_1$ and $T_2$ are compatible $\mathcal{O}$-operators, we have
\begin{align*}
[T_1 (m), T_2 (n)] + [T_2 (m), T_1(n)] = T_1 ( T_2(m) \bullet n - T_2(n) \bullet m) + T_2 ( T_1 (m) \bullet n - T_1(n) \bullet m).
\end{align*} 
Since $T_1 = T_2 \circ S$,
\begin{align}\label{eqn-pp}
[T_2S (m), T_2 (n)] + [T_2 (m), T_2S(n)] = T_2S ( T_2(m) \bullet n - T_2(n) \bullet m) + T_2 ( T_2S (m) \bullet n - T_2S(n) \bullet m).
\end{align}
On the other hand, $T_2$ is an $\mathcal{O}$-operator implies that
\begin{align}\label{eqn-qq}
[T_2S (m), T_2 (n)] + [T_2 (m), T_2S(n)]  = T_2 \big( T_2 S(m) \bullet n - T_2 (n) \bullet S(m) + T_2 (m) \bullet S(n) - T_2 S(n) \bullet m \big).
\end{align}
From (\ref{eqn-pp}) and (\ref{eqn-qq}) and using the fact that $T_2$ is invertible, we get
\begin{align}\label{eqn-rr}
S ( T_2 (m) \bullet n - T_2(n) \bullet m ) = T_2(m) \bullet S(n) - T_2(n) \bullet S(m).
\end{align}
By replacing $n$ by $S(n)$,
\begin{align}\label{eqn-ss}
T_2 (m) \bullet S^2 (n) - S ( T_2 (m) \bullet S(n)) = - S ( T_2 S(n) \bullet m) + T_2 S(n) \bullet S(m).
\end{align}
As $T_1 = T_2 \circ S$ and $T_2$ are $\mathcal{O}$-operators,
\begin{align*}
T_2 \circ S ( [m,n]^{T_2 \circ S} ) = [T_2 S (m), T_2 S (n) ] = T_2 ( [S(m), S(n)]^{T_2}).
\end{align*}
The invertibility of $T_2$ implies that $S ( [m,n]^{T_2 \circ S} ) = [S(m), S(n)]^{T_2}$, or, equivalently,
\begin{align}\label{eqn-tt}
 S ( T_2 S(m) \bullet n - T_2 S(n) \bullet m) = T_2 S(m) \bullet S(n) - T_2 S(n) \bullet S(m).
\end{align}
Finally, from (\ref{eqn-ss}) and (\ref{eqn-tt}), we get
\begin{align*}
T_2 (m) \bullet S^2(n) - S ( T_2(m) \bullet S(n)) = T_2 S(m) \bullet S(n) - S ( T_2 S(m) \bullet n).
\end{align*} 
Substitute $x = T_2 (m)$, using $T_2 S = N T_2$ and the invertibility of $T_2$,
\begin{align*}
x \bullet S^2 (n) - S ( x \bullet S(n)) = N (x) \bullet S(n) - S ( N(x) \bullet n).
\end{align*}
Hence the identity (\ref{nij-second}) follows. Thus, the pair $(N, S)$ is a Nijenhuis structure on $M$ over $\mathfrak{g}$.

Next, observe that $N \circ T_2 = T_2 \circ S = T_1$. Moreover,
\begin{align*}
[m,n]^{T_2}_S - [m,n]^{T_2 \circ S} = T_2 (m) \bullet S(n) - T_2 (n) \bullet S(m) - S (T_2 (m) \bullet n - T_2 (n) \bullet m) = 0 ~~(\text{by } (\ref{eqn-rr})).
\end{align*}
Hence $(T_2 , N = T_1 \circ T_2^{-1}, S= T_2^{-1} \circ T_1)$ is an $\mathcal{ON}$-structure on $M$ over $\mathfrak{g}$.
\end{proof}

Let $(T, N, S)$ be an $\mathcal{ON}$-structure on $M$ over the Lie algebra $\mathfrak{g}$. For any $k \geq 0$, we define $T_k := T \circ S^k = N^k \circ T$. The next lemma is analogous to the similar result for Poisson-Nijenhuis structures \cite{ks-magri}.

\begin{lemma}
For all $k, l \geq 0$, we have
\begin{align}
T_k ( [m, n]^T_{S^{k+l}}) =~& [T_k (m), T_k (n)]_{N^l}, \label{first-k}\\
[m,n]^{T_{k+l}} =~& [m,n]^T_{S^{k+l}} = S^k ([m,n]^{T_l}). \label{first-kl}
\end{align}
\end{lemma}

\begin{thm}\label{final-thm-comp-o}
Let $(T, N, S)$ be an $\mathcal{ON}$-structure on $M$ over the Lie algebra $\mathfrak{g}$. Then for all $k \geq 0$, the linear maps $T_k$ are $\mathcal{O}$-operators. Moreover, for all $k, l \geq 0$, the $\mathcal{O}$-operators $T_k$ and $T_l$ are compatible.
\end{thm}

\begin{proof}
We have from (\ref{first-k}) and (\ref{first-kl}) that
\begin{align*}
T_k ([m,n]^{T_k} ) = T_k ( [m,n ]^T_{S^k}) = [ T_k (m), T_k (n)]
\end{align*}
which shows that $T_k$ is an $\mathcal{O}$-operator on $M$ over $\mathfrak{g}$. To prove the second part, we first observe that
\begin{align*}
[m, n]^{T_k + T_{k+l}} = [m,n]^{T_k} + [m, n]^{T_{k+l}} = [m,n]^{T_k} + [m,n]^{T_k}_{S^l} ~~ (\text{by } (\ref{first-kl})).
\end{align*}
Hence
\begin{align*}
&(T_k + T_{k+l}) ([m,n]^{T_k + T_{k+l}} ) \\
& = T_k ([m,n]^{T_k} ) + T_k ([m,n]^{T_k}_{S^l}) + T_{k+l} ([m,n]^{T_k} ) + T_{k+l} ([m,n]^{T_k}_{S^l}) \\
&= [ T_k (m), T_k (n)] + T_k ( [S^l(m), n]^{T_k} + [ m, S^l (n)]^{T_k} - S^l [m,n]^{T_k} ) + T_k \circ S^l ([m,n]^{T_k}) + T_{k+l} ([m,n]^{T_k}_{S^l} ) \\
& = [ T_k (m), T_k (n)] + [T_{k+l} (m), T_k (n)] + [ T_k (m), T_{k+l} (n) ]+ [ T_{k+l}(m), T_{k+l} (n) ]\\
& = [ (T_k + T_{k+l})  (m), (T_k + T_{k+l}) (n)].
\end{align*}
This shows that $T_k + T_{k+l}$ is an $\mathcal{O}$-operator. Hence $T_k$ and $T_{k+l}$ are compatible.
\end{proof}

\medskip

We have mentioned earlier that classical r-matrices are Lie algebra analog of Poisson structures. Here we mention Lie algebra analog of Poisson-Nijenhuis structures and their relation with $\mathcal{ON}$-structures.

\begin{defn}
Let $\mathfrak{g}$ be a Lie algebra. A pair $({\bf r}, N)$ consisting of a classical r-matrix and a Nijenhuis tensor on $\mathfrak{g}$ is called a PN-structure on $\mathfrak{g}$ if they satisfy
\begin{itemize}
\item[$\triangleright$] $N \circ {\bf r}^\sharp = {\bf r}^\sharp \circ N^*,$
\item[$\triangleright$] $[\alpha, \beta ]^{N \circ {\bf r}^\sharp} = [\alpha, \beta ]^{{\bf r}^\sharp}_{N^*}$, for all $\alpha, \beta \in \mathfrak{g}^*$.
\end{itemize}
\end{defn}

The next result generalizes Lemma \ref{lemma-r-mat} in the realm of Poisson-Nijenhuis structures.

\begin{prop}
Let $({\bf r}, N)$ be a pair consisting of an element ${\bf r} \in \wedge^2 \mathfrak{g}$ and a linear map $N$ on $\mathfrak{g}$. Then $({\bf r}, N)$ is a PN-structure on $\mathfrak{g}$ if and only if the triple $({\bf r}^\sharp, N, N^*)$ is an $\mathcal{ON}$ structure on the coadjoint representation $\mathfrak{g}^*$ over the Lie algebra $\mathfrak{g}$.
\end{prop}

Thus, by Theorem \ref{nt}, Proposition \ref{nt-com} and Theorem \ref{final-thm-comp-o}, we obtain the following.

\begin{corollary}
Let $({\bf r}, N)$ be a PN-structure on a Lie algebra $\mathfrak{g}$. Then ${\bf r}_N \in \wedge^2 \mathfrak{g}$ defined by $({\bf r}_N)^\sharp := N \circ {\bf r}^\sharp$ is a classical r-matrix. Moreover, the classical r-matrices ${\bf r}$ and ${\bf r}_N$ are compatible in the sense that their linear combinations are also classical r-matrices.
\end{corollary}

\begin{corollary}
Let $({\bf r}, N)$ be a PN-structure on a Lie algebra $\mathfrak{g}$. Then for all $k \geq 0$, the elements ${\bf r}_k \in \wedge^2 \mathfrak{g}$ defined by $({\bf r}_k)^\sharp := N^k \circ {\bf r}^\sharp$ are classical r-matrices. Moreover, the classical r-matrices ${\bf r}_k$ and ${\bf r}_l$ are pairwise compatible.
\end{corollary}




%

\section{Strong Maurer-Cartan equation on a twilled Lie algebra and $\mathcal{ON}$-structures}
In this section, we construct a twilled Lie algebra from an $\mathcal{O}$-operator. Then we associate an $\mathcal{ON}$-structure to any solution of the strong Maurer-Cartan equation on that twilled Lie algebra.

Let $\mathfrak{g}$ be a Lie algebra and $\mathfrak{a, b} \subset \mathfrak{g}$ be two subspaces satisfying $\mathfrak{g} = \mathfrak{a} \oplus \mathfrak{b}.$

\begin{defn}
A triple $(\mathfrak{g}, \mathfrak{a}, \mathfrak{b})$ is called a twilled Lie algebra if $\mathfrak{a}$ and $\mathfrak{b}$ are Lie subalgebras of $\mathfrak{g}$. We also denote a twilled Lie algebra by $\mathfrak{a} \bowtie \mathfrak{b}.$
\end{defn}

Let $(\mathfrak{g}, \mathfrak{a}, \mathfrak{b})$ be a twilled Lie algebra. Then there are Lie algebra representations $\bullet_1 : \mathfrak{a} \times \mathfrak{b} \rightarrow \mathfrak{b}$ and $\bullet_2 : \mathfrak{b} \times \mathfrak{a} \rightarrow \mathfrak{a}$ given by the following decomposition
\begin{align*}
[x, u] = x \bullet_1 u - u \bullet_2 x, ~\text{ for } x \in \mathfrak{a}, u \in \mathfrak{b}.
\end{align*}
Consider the semi-direct product Lie algebra associated to the action $\bullet_2$. Let $\mu_2$ denote the corresponding multiplication map.

Note that the graded vector space $C^{* + 1}_{\mathrm{CE}} (\mathfrak{g}, \mathfrak{g} ) = \oplus_{n \geq 0} C^{n+1}_{\mathrm{CE}} (\mathfrak{g}, \mathfrak{g}) = \oplus_{n \geq 0} \mathrm{Hom}(\wedge^{n+1} \mathfrak{g}, \mathfrak{g} )$ carries a graded Lie algebra structure with the Nijenhuis-Richardson bracket $\{ P, Q \} = P \odot Q - (-1)^{pq} Q \odot P$, for $P \in C^{p+1} (\mathfrak{g}, \mathfrak{g})$, $Q \in C^{q+1} (\mathfrak{g}, \mathfrak{g})$, where $P \odot Q$ is given by
\begin{align*}
(P \odot Q) (x_1, \ldots, x_{p+q+1}) = \sum_{\tau \in \mathrm{Sh}(q+1, p)} \mathrm{sgn}(\sigma)~ P ( Q (x_{\tau(1)}, \ldots, x_{\tau (q+1)}), x_{\tau (q+2)}, \ldots, x_{\tau (p+q+1)} ).
\end{align*}

Consider the graded space $C^*_{\mathrm{CE}} ( \mathfrak{a}, \mathfrak{b}) = \oplus_{n \geq 0} C^n_{\mathrm{CE}} ( \mathfrak{a}, \mathfrak{b}) = \oplus_{n \geq 0} \mathrm{Hom}( \wedge^n \mathfrak{a}, \mathfrak{b})$ with the Chevalley-Eilenberg differential $d_{\mathrm{CE}} : C^n_{\mathrm{CE}} ( \mathfrak{a}, \mathfrak{b}) \rightarrow C^{n+1}_{\mathrm{CE}} ( \mathfrak{a}, \mathfrak{b})$ for the representation of the Lie algebra $\mathfrak{a}$ on $(\mathfrak{b}, \bullet_1)$. This graded space also carries a graded Lie algebra structure via the derived bracket (see \cite{voro})
\begin{align}\label{mu-2-brckt}
[P, Q]_{\mu_2} := (-1)^{p-1} \{ \{ \mu_2, P \}, Q \}, ~\text{for } P \in C^p_{\mathrm{CE}} (\mathfrak{a}, \mathfrak{b}), ~ Q \in C^q_{\mathrm{CE}} (\mathfrak{a}, \mathfrak{b}). 
\end{align}
This two structures are compatible in the sense that $(C^*_{\mathrm{CE}} ( \mathfrak{a}, \mathfrak{b}), d_{\mathrm{CE}}, [~, ~]_{\mu_2})$ is a dgLa \cite{voro}.

\begin{defn}
Let $\mathfrak{a} \bowtie \mathfrak{b}$ be a twilled Lie algebra. An element $\Omega \in C^1_{\mathrm{CE}} ( \mathfrak{a}, \mathfrak{b})$ is called a solution of the Maurer-Cartan equation if it satisfies
\begin{align*}
d_{\mathrm{CE}} \Omega + \frac{1}{2} [\Omega, \Omega]_{\mu_2} = 0.
\end{align*}
It is called a solution of the strong Maurer-Cartan equation if $d_{\mathrm{CE}} \Omega  = \frac{1}{2} [\Omega, \Omega]_{\mu_2} = 0$.
\end{defn}

\begin{lemma}\label{char-mc-smc}
Let $\mathfrak{a} \bowtie \mathfrak{b}$ be a twilled Lie algebra and $\Omega \in  C^1_{\mathrm{CE}} ( \mathfrak{a}, \mathfrak{b}) $. Then
\begin{itemize}
\item[(i)] $\Omega$ is a solution of the Maurer-Cartan equation if and only if $\Omega$ satisfies
\begin{align}\label{mc-equiv}
[\Omega (x), \Omega(y)] + x \bullet_1 \Omega (y) - y \bullet_1 \Omega (x) = \Omega ( \Omega (x) \bullet_2 y - \Omega(y) \bullet_2 x ) + \Omega ([x, y]),~\text{for } x, y \in \mathfrak{a}. 
\end{align}
\item[(ii)] $\Omega$ is a solution of the strong Maurer-Cartan equation if and only if $\Omega$ satisfies (\ref{mc-equiv}) and
\begin{align*}
\Omega ([x, y]) = x \bullet_1 \Omega (y) - y \bullet_1 \Omega (x).
\end{align*}
\end{itemize}
\end{lemma}

\begin{proof}
Note that
\begin{align*}
(d_{\mathrm{CE}} \Omega) (x, y) = x \bullet_1 \Omega (y) - y \bullet_1 \Omega (x) - \Omega ([x, y]).
\end{align*}
From the definition of the bracket (\ref{mu-2-brckt}), it is easy to see that
\begin{align*}
[\Omega, \Omega ]_{\mu_2} (x, y) = 2 \big( [\Omega (x), \Omega (y)] - \Omega ( \Omega (x) \bullet_2 y) + \Omega (\Omega(y) \bullet_2 x) \big).
\end{align*}
Hence the result follows from the definition of the (strong) Maurer-Cartan equation.
\end{proof}

Next, let $T : M \rightarrow \mathfrak{g}$ be an $\mathcal{O}$-operator on $M$ over the Lie algebra $\mathfrak{g}$. Consider the Lie algebra $M^T$. Then it has been shown in \cite{tang-et} that the map $\overline{\bullet} : M \times \mathfrak{g} \rightarrow \mathfrak{g}$ given by
\begin{align*}
m ~\overline{\bullet}~  x = [T(m), x] + T (x \bullet m), ~ \text{for } m \in M, x \in \mathfrak{g}
\end{align*}
defines a representation of the Lie algebra $M^T$ on the vector space $\mathfrak{g}$.

Consider the direct sum $\mathfrak{g} \oplus M^T$ with the bracket
\begin{align*}
[[ (x , m), (y , n) ]]^T = ([x, y] +  m ~\overline{\bullet}~ y - n ~\overline{\bullet}~ x ,~ x \bullet n - y \bullet m + [m, n]^T).
\end{align*}
Hence by Proposition \ref{char-mc-smc}, we have the following.
\begin{thm}\label{thm-ee}
The vector space $\mathfrak{g} \oplus M^T$ with the above bracket $[[~, ~]]^T$ is a twilled Lie algebra (denoted by $\mathfrak{g} \bowtie M^T$). A linear map $\Omega : \mathfrak{g} \rightarrow M$ is a solution of the strong Maurer-Cartan equation on the twilled Lie algebra $\mathfrak{g} \bowtie M^T$ if and only if $\Omega$ satisfies
\begin{align}
\Omega [x, y] =~&  x \bullet \Omega (y) - y \bullet \Omega (x), \label{eqn-aee}\\
[\Omega (x), \Omega (y) ]^T =~& \Omega ( \Omega (x) \overline{\bullet} y - \Omega(y) \overline{\bullet} x), ~ \text{for } x, y \in \mathfrak{g}. \label{eqn-bee}
\end{align}
\end{thm}

It follows from (\ref{eqn-bee}) that $\Omega$ is an $\mathcal{O}$-operator on the module $\mathfrak{g}$ over the Lie algebra $M^T$. Thus, $\Omega$ induces a new Lie algebra structure on $\mathfrak{g}$, denoted by $\mathfrak{g}^\Omega$, the corresponding Lie bracket $[~, ~]^\Omega$ is given by
\begin{align*}
[x, y]^\Omega = \Omega (x) ~\overline{\bullet}~ y - \Omega(y) ~\overline{\bullet}~ x, ~ \text{for } x, y \in \mathfrak{g}.
\end{align*}

This Lie algebra has a representation on $M$ given by
$x \bullet^\Omega m = [ \Omega (x), m ]^T + \Omega ( m ~\overline{\bullet}~ x )$, for $x \in \mathfrak{g}$, $m \in M$.
Therefore, we may define a new bracket on $\mathfrak{g} \oplus M$ by
\begin{align*}
[ (x , m), (y , n)]^\Omega_T :=  (m ~\overline{\bullet}~ y - n ~\overline{\bullet}~ x + [x, y]^\Omega , x \bullet^\Omega n - y \bullet^\Omega m + [m, n]^T).
\end{align*}

\begin{thm}
Let $\Omega : \mathfrak{g} \rightarrow M$ be a solution of the strong Maurer-Cartan equation on the twilled Lie algebra $\mathfrak{g} \bowtie M^T$. Then
\begin{itemize}
\item[(i)] $(\mathfrak{g} \oplus M, [~,~]^\Omega_T)$ is a Lie algebra (we denote the corresponding twilled Lie algebra by $\mathfrak{g}^\Omega \bowtie M^T$).
\item[(ii)] $T$ is a solution of the strong Maurer-Cartan equation on the twilled Lie algebra $M^T \bowtie \mathfrak{g}^\Omega$.
\item[(iii)] $T$ is an $\mathcal{O}$-operator on the module $(M, \bullet^\Omega)$ over the Lie algebra $\mathfrak{g}^\Omega$.
\end{itemize}
\end{thm}

\begin{proof} (i) It follows from a direct verification and by using Theorem \ref{thm-ee}.

(ii) To prove that $T$ is a solution of the strong Maurer-Cartan equation on the twilled Lie algebra $M^T \bowtie \mathfrak{g}^\Omega$, by Theorem \ref{thm-ee}, one needs to verify that
\begin{align*}
T ([m,n]^T) = m ~\overline{\bullet}~ T(n) - n ~\overline{\bullet}~ T(m) ~~~ \text{ and } ~~~ [Tm, Tn]^\Omega = T (T(m) \bullet^\Omega n - T(n) \bullet^\Omega m).
\end{align*}
Observe that
\begin{align*}
T ([m,n]^T) = [Tm, Tn] ~&= [Tm, Tn] + T (T(n) \bullet m) - [Tn, Tm] - T (T(m) \bullet n) \\
~&= m ~\overline{\bullet}~ T(n) - n ~\overline{\bullet}~ T(m).
\end{align*}
On the other hand,
\begin{align*}
[Tm, Tn]^\Omega =~& \Omega T(m) ~\overline{\bullet}~ T(n) - \Omega T(n) ~\overline{\bullet}~ T(m) \\
=~& [T \Omega T (m), T(n)] + T ( T(n) \bullet \Omega T(m)) - [ T \Omega T (n), T(m) ] - T ( T(m) \bullet \Omega T(n)) \\
=~& [ T \Omega T (m), T(n) ] - [ T \Omega T (n), T(m) ] - T \Omega [ T(m), T(n) ] \quad (\text{by } (\ref{eqn-aee})) \\
=~& [ T \Omega T (m), T(n) ] - [ T \Omega T (n), T(m)]  - T \Omega ( m \overline{\bullet} T(n) - n \overline{\bullet} T(m)) \\
=~& T \big( [ \Omega T (m), n]^T + \Omega (n ~\overline{\bullet}~ T(m)) - [ \Omega T (n), m]^T - \Omega ( m ~\overline{\bullet}~ T(n)) \big) \\
=~& T ( T(m) \bullet^\Omega n - T(n) \bullet^\Omega m ).
\end{align*}

(iii) By a direct calculation, one can show that
\begin{align*}
T ( T(m) \bullet^\Omega n - T(n) \bullet^\Omega m ) =~& T \big(  T \Omega T (m) \bullet n - T(n) \bullet \Omega T (m) + T ( T(n) \bullet \Omega T (m)) \\ & - T \Omega T(n) \bullet m + T(m) \bullet \Omega T (n) - T (T(m) \bullet \Omega T(n) ) \big).
\end{align*}
Hence
\begin{align*}
T ( T(m) \bullet^\Omega n - T(n) \bullet^\Omega m ) =~& [T \Omega T (m), T(n) ] + T (T(n) \bullet \Omega T (m)) - [ T \Omega T (n), T(m)] - T (T(m) \bullet \Omega T (n))\\
=~& \Omega T (m) ~\overline{\bullet }~ T(n) - \Omega T (n) ~\overline{\bullet}~ T(m) = [ T(m), T(n)]^\Omega .
\end{align*}
This proves that $T$ is an $\mathcal{O}$-operator on the module $(M, \bullet^\Omega)$ over $\mathfrak{g}^\Omega$.
\end{proof}

We are now in a position to prove our main result of this section.

\begin{thm}\label{final-thm-strong}
Let $T : M \rightarrow \mathfrak{g}$ be an $\mathcal{O}$-operator on a module $M$ over the Lie algebra $\mathfrak{g}$. If $\Omega : \mathfrak{g} \rightarrow M$ is a solution of the strong Maurer-Cartan equation on the twilled Lie algebra $\mathfrak{g} \bowtie M^T$ then
 $(T, N = T \circ \Omega, S = \Omega \circ T)$ is an $\mathcal{ON}$-structure on the module $M$ over the Lie algebra $\mathfrak{g}$.
\end{thm}

\begin{proof}
For any $x, y \in \mathfrak{g}$, we have
\begin{align*}
[T \Omega (x), T \Omega (y)] =~& T ([\Omega x, \Omega y]^T) \\
=~& T \Omega ( \Omega (x) ~\overline{\bullet}~ y - \Omega (y ) ~\overline{\bullet}~ x) ~~~ (\text{by } (\ref{eqn-bee})) \\
=~& T \Omega ( [ T \Omega (x), y] + T ( y \bullet \Omega (x)) + [x, T\Omega (y)] - T (x \bullet \Omega (y)) ) \\
=~& T \Omega ( [ T \Omega (x), y] + [x, T\Omega (y)] - T \Omega [x, y]) ~~~ (\text{by } (\ref{eqn-aee})).
\end{align*}
This shows that $N = T \Omega$ is a Nijenhuis tensor on $\mathfrak{g}$. We will now show that $(N, S)$ is a Nijenhuis structure on the module $M$. First observe that
\begin{align}
\Omega T ( T(m) \bullet n - T(n) \bullet m) = \Omega [ Tm, Tn] = T(m) \bullet \Omega Tn - T(n) \bullet \Omega T(m) ~~~ (\text{by } (\ref{eqn-aee})). \label{star}
\end{align}
On the other hand, by taking $y = T(n)$ in (\ref{eqn-bee}), we get
\begin{align}
[ \Omega x, \Omega T(n)]^T =~& \Omega \big( [T \Omega (x), T(n)] + T ( T(n) \bullet \Omega (x)) - [T \Omega T (n), x] - T ( x \bullet \Omega T (n))   \big), \nonumber \\
\text{or,} ~~~T \Omega (x) \bullet \Omega T(n) - T \Omega T(n) \bullet \Omega (x) =~& T \Omega (x) \bullet \Omega T(n) - T(n) \bullet \Omega T \Omega(x) + \Omega T ( T(n) \bullet \Omega(x)) \nonumber \\
~&-T \Omega T(n) \bullet \Omega (x) + x \bullet \Omega T \Omega T (n) - \Omega T (x \bullet \Omega T(n)). \label{double-star}
\end{align}
By (\ref{star}) and (\ref{double-star}), we get
\begin{align*}
T \Omega (x) \bullet \Omega T(n) - \Omega T ( T \Omega (x) \bullet n ) = x \bullet \Omega T \Omega T (n) - \Omega T ( x \bullet \Omega T(n)).
\end{align*}
This is just equation (\ref{nij-second}) for $N = T \Omega$ and $S = \Omega T$. Therefore $(N, S)$ is a Nijenhuis structure on the $\mathfrak{g}$-module $M$.

We also have $T \circ S = N \circ T = T \circ \Omega \circ T$. Finally, a direct calculation shows that
\begin{align*}
[m, n]^T_S - [m,n]^{T \circ S} =~& T (m ) \bullet S(n) - T(n) \bullet S(m) - S ( T(m) \bullet n - T(n) \bullet m) \\
=~& T(m) \bullet \Omega T(n) - T(n) \bullet \Omega T(m) - \Omega T  ( T(m) \bullet n - T(n) \bullet m) = 0 ~~(\text{by } (\ref{star})).
\end{align*}
Hence $(T, N, S)$ is a Nijenhuis structure on $M$ over $\mathfrak{g}$.
\end{proof}

Thus, in view of Theorem \ref{nt}, Proposition \ref{nt-com} and Theorem \ref{final-thm-comp-o}, we get the following.

\begin{corollary}\label{cor-strong-mc}
Let $T$ be an $\mathcal{O}$-operator on $M$ over the Lie algebra $\mathfrak{g}$. If $\Omega : \mathfrak{g} \rightarrow M$ is a solution of the strong Maurer-Cartan equation on the twilled Lie algebra $\mathfrak{g} \bowtie M^T$, then for all $k \geq 0$, $T_k := (T \circ \Omega)^k \circ T$ are $\mathcal{O}$-operators on $M$ over $\mathfrak{g}$ and they are pairwise compatible.
\end{corollary}

In the above theorem, we show that given an $\mathcal{O}$-operator $T$, a solution $\Omega$ of the strong Maurer-Cartan equation on the twilled Lie algebra $\mathfrak{g} \bowtie M^T$ leads to an $\mathcal{ON}$-structure $(T, N = T \circ \Omega, S = \Omega \circ T)$. The converse of the above theorem holds true provided $T$ is invertible. 

\begin{thm}\label{thm-strong-last}
Let $(T, N, S)$ be an $\mathcal{ON}$-structure on $M$ over the Lie algebra $\mathfrak{g}$, in which $T$ is invertible. Then $\Omega := T^{-1} \circ N = S \circ T^{-1} : \mathfrak{g} \rightarrow M$ is a solution of the strong Maurer-Cartan equation on the twilled Lie algebra $\mathfrak{g} \bowtie M^T$.
\end{thm}

\begin{proof}
Since $N = T \circ \Omega$ is a Nijenhuis tensor,
\begin{align*}
[T \Omega (x), T \Omega (y)] = T \Omega ( [T \Omega (x), y] + [ x, T \Omega (y)] - T \Omega [x, y]).
\end{align*}
On the other hand,
\begin{align*}
\Omega (x) ~\overline{\bullet}~ y - \Omega (y) ~ \overline{\bullet}~ x =~& [ T \Omega (x), y] + T (y \bullet \Omega (x)) - [T \Omega (y), x] - T ( x \bullet \Omega (y)) \\
=~& [ T \Omega (x), y] + [x, T \Omega (y)] - T \Omega [x, y].
\end{align*}
Since $T$ is an $\mathcal{O}$-operator we have
\begin{align*}
T ([\Omega (x), \Omega (y)]^T) = [T \Omega (x), T \Omega (y)] = T \Omega ( \Omega (x) ~ \overline{\bullet}~ y - \Omega (y) ~\overline{\bullet}~ x ).
\end{align*}
Hence the equation (\ref{eqn-bee}) follows as $T$ is invertible.

On the other hand, from $[m,n]^T_S = [m, n]^{TS}$ with $S = \Omega \circ T$, we deduce that
\begin{align*}
\Omega T ( T(m) \bullet n - T(n) \bullet m) = T(m) \bullet \Omega T(n) - T(n) \bullet \Omega T(m).
\end{align*}
In other words,
\begin{align*}
\Omega [ T(m), T(n)] = T(m) \bullet \Omega T(n) - T(n) \bullet \Omega T(m).
\end{align*}
Hence the equation (\ref{eqn-aee}) follows by taking $T(m)= x$ and $T(n) = y$. Therefore, $\Omega $ is a solution of the strong Maurer-Cartan equation by Theorem \ref{thm-ee}.
\end{proof}

\section{Generalized complex structures}\label{sec-gcs}
In this section, we introduce generalized complex structures on a module over a Lie algebra. When one consider the coadjoint module $\mathfrak{g}^*$ of a Lie algebra $\mathfrak{g}$, one get generalized complex structures on $\mathfrak{g}$.

Let $\mathfrak{g}$ be a Lie algebra and $M$ be a $\mathfrak{g}$-module. Consider the semi-direct product Lie algebra $\mathfrak{g} \ltimes M$ with the bracket given in (\ref{semi-dir-brkt}). Let $J: \mathfrak{g} \oplus M \rightarrow \mathfrak{g} \oplus M$ be a linear map. Then $J$ must be of the form
\begin{align}\label{gcs-o}
J = \left( \begin{array}{cc}
N    &      T\\
\sigma     &    -S
\end{array}
\right),
\end{align}
for some linear maps $N : \mathrm{End} (\mathfrak{g}), S \in \mathrm{End} (M),~ T : M \rightarrow \mathfrak{g}$ and $\sigma : \mathfrak{g} \rightarrow M$. These linear maps are called structure components of $J$. The reason behind considering the linear map as $-S$ (instead of $S$) will be clear from Proposition \ref{gcs-o-gcs-l}.

\begin{defn}\label{defn-gcs-o}
A generalized complex structure on $M$ over the Lie algebra $\mathfrak{g}$ is a linear map $J : \mathfrak{g} \oplus M \rightarrow \mathfrak{g} \oplus M$ satisfying the following conditions
\begin{itemize}
\item[(i)] $J$ is almost complex: ~ $J^2 = - \mathrm{id}$,
\item[(ii)] integrability condition: ~
$[Ju, Jv] - [u,v] - J ([Ju, v]+ [u, Jv]) = 0, ~ \text{for } u, v \in \mathfrak{g} \oplus M.$
\end{itemize}
\end{defn}

In \cite{crainic} Crainic gives a characterization of generalized complex manifolds. A similar theorem in our context reads as follows.

\begin{thm}\label{gcs-o-char}
A linear map $J :  \mathfrak{g} \oplus M \rightarrow \mathfrak{g} \oplus M$ of the form (\ref{gcs-o}) is a generalized complex structure on $M$ over the Lie algebra $\mathfrak{g}$ if and only if its structure components satisfy the following identities:
\begin{align}
NT =~& TS,\label{gcs-equiv-1}\\
N^2 + T \sigma =~& -\mathrm{id},\\
S \sigma =~& \sigma N, \\
S^2 + \sigma T =~& - \mathrm{id}, \label{gcs-equiv-4}\\
T ([m,n]^T) =~& [Tm, Tn], \label{gcs-equiv-5}\\
S ([m,n]^T)= ~& Tm \bullet Sn - Tn \bullet Sm, \label{gcs-equiv-6}\\
[Nx, Tm] - N [x, Tm] =~& T ( Nx \bullet m - x \bullet Sm), \label{gcs-equiv-7}\\
\sigma [Tm, x] - Tm \bullet \sigma x =~& x \bullet m + Nx \bullet Sm - S ( Nx \bullet m - x \bullet Sm), \label{gcs-equiv-8}\\
[Nx, Ny] - [x, y] - N ([Nx, y]+ [x, Ny]) =~& T ( x \bullet \sigma y - y \bullet \sigma x), \label{gcs-equiv-9}\\
Nx \bullet \sigma y - Ny \bullet \sigma x - \sigma ([Nx, y]+ [x, Ny])=~& - S (x \bullet \sigma y - y \bullet \sigma x). \label{gcs-equiv-10}
\end{align}
\end{thm}

\begin{proof}
The condition $J^2 = - \mathrm{id}$ is same as
\begin{align*}
\left( \begin{array}{c}
N^2(x) + NT (x) + T \sigma (x) - TS (m)  \\
\sigma N (x) + \sigma T(m) - S \sigma (x) + S^2 (m)    
\end{array}
\right) = \left( \begin{array}{c}
-x  \\
-m   
\end{array}
\right).
\end{align*}
This is equivalent to the identities (\ref{gcs-equiv-1})-(\ref{gcs-equiv-4}). Next consider the integrability condition of  $J$. For $u=m, ~ v = n \in M$, we get from the integrability criteria that (\ref{gcs-equiv-5}) and (\ref{gcs-equiv-6}) holds. For $u = x \in \mathfrak{g}$ and $v = m \in M$, the integrability is equivalent to (\ref{gcs-equiv-7}) and (\ref{gcs-equiv-8}). Finally, for $u = x ,~ v = y \in \mathfrak{g}$, we get the identities (\ref{gcs-equiv-9}) and (\ref{gcs-equiv-10}).
\end{proof}

\begin{remark}
Note that the condition (\ref{gcs-equiv-5}) implies that the map $T: M \rightarrow \mathfrak{g}$ is an $\mathcal{O}$-operator on $M$ over $\mathfrak{g}$. The condition (\ref{gcs-equiv-6}) is equivalent to the fact that
\begin{align*}
\mathrm{Gr}((T,S)) = \{ ( T(m), S(m))|~ m \in M \} \subset \mathfrak{g} \oplus M
\end{align*}
is a subalgebra of the semi-direct product $\mathfrak{g} \ltimes M$.
\end{remark}

The above Theorem ensures the following examples of generalized complex structures on modules over Lie algebras.

\begin{exam} (Opposite g.c.s.) Let $J = \left( \begin{array}{cc}
N   &      T\\
\sigma     &   -S
\end{array}
\right)$ be a generalized complex structure on $M$ over $\mathfrak{g}$. Then $\overline{J} = \left( \begin{array}{cc}
N   &      - T\\
- \sigma     &   -S
\end{array}
\right)$ is also a generalized complex structure called the opposite of $J$.
\end{exam}

\begin{exam}
Let $T: M \rightarrow \mathfrak{g}$ be an invertible $\mathcal{O}$-operator on $M$ over $\mathfrak{g}$. Then $J = \left( \begin{array}{cc}
0   &      T\\
- T^{-1}     &   0
\end{array}
\right)$ is a generalized complex structure on $M$ over $\mathfrak{g}$.
\end{exam}

\begin{defn}
A complex structure on a Lie algebra $\mathfrak{g}$ is a linear map $I : \mathfrak{g} \rightarrow \mathfrak{g}$ satisfying $I^2 = - \mathrm{id}$ and
\begin{align*}
[Ix, Iy] - [x, y] - I ( [Ix, y]+[x, Iy]) = 0, ~ \text{for } x, y \in \mathfrak{g}.
\end{align*}
\end{defn}

\begin{defn}
Let $\mathfrak{g}$ be a Lie algebra and $M$ be a $\mathfrak{g}$-module. A complex structure on $M$ over the Lie algebra $\mathfrak{g}$ is a pair $(I, I_M)$ of linear maps $I \in \mathrm{End}(\mathfrak{g})$ and $I_M \in \mathrm{End}(M)$ satisfying the followings
\begin{itemize}
\item[$\triangleright$] $I$ is a complex structure on $\mathfrak{g}$,
\item[$\triangleright$] $I_M^2 = - \mathrm{id}$ and  \begin{equation}\label{iden-s}
I(x) \bullet I_M (m) - x \bullet m - I_M ( I(x) \bullet m + x \bullet I_M (m)) = 0, \text{ for }  x \in \mathfrak{g}, m \in M.
\end{equation}
\end{itemize}
\end{defn}

\begin{prop}
A pair $(I, I_M)$ is a complex structure on $M$ over the Lie algebra $\mathfrak{g}$ if and only if $I \oplus I_M : \mathfrak{g} \oplus M \rightarrow \mathfrak{g} \oplus M$ is a complex structure on the semi-direct Lie algebra $\mathfrak{g} \ltimes M$.
\end{prop}

\begin{proof}
The linear map $I \oplus I_M$ is a complex structure on the semi-direct product $\mathfrak{g} \ltimes M$ if and only if $(I \oplus I_M)^2 = - \mathrm{id}$ (equivalently, $I^2 = -\mathrm{id}$ and $I_M^2 = - \mathrm{id})$ and
\begin{align*}
[ (Ix, I_M m), (Iy, I_M m)] - [(x,m), (y, n)] - (I \oplus I_M ) ( [ (Ix, I_M m), (y, n)] + [(x, m), (Iy, I_M n)] ) = 0,
\end{align*}
or, equivalently, 
\begin{align*}
\big( [Ix, Iy] - [x, y] - I ([Ix, y] + [x, Iy]), ~& Ix \bullet I_M n - Iy \bullet I_M n - (x \bullet n - y \bullet m) \\ &- I_M (Ix \bullet n - y \bullet I_M m + x \bullet I_M n - Iy \bullet m) \big) = 0.
\end{align*}
The last identity is equivalent to the fact that $I$ is a complex structure on $\mathfrak{g}$ and the identity (\ref{iden-s}) holds. In other words $(I, I_M)$ is a complex structure on $M$ over the Lie algebra $\mathfrak{g}$.
\end{proof}

\begin{exam}
Let $(I, I_M)$ be a complex structure on $M$ over the Lie algebra $\mathfrak{g}$. Then $J = \left( \begin{array}{cc}
I   &      0\\
0     &   I_M
\end{array}
\right)$ is a generalized complex structure on $M$ over $\mathfrak{g}$. Note that here $S= - I_M.$
\end{exam}


Next we consider generalized complex structures on Lie algebras and show that they are related to generalized complex structures on coadjoint modules.
Let $\mathfrak{g}$ be a Lie algebra. Consider the coadjoint representation $\mathfrak{g}^*$ and the corresponding semi-direct Lie algebra $\mathfrak{g} \ltimes \mathfrak{g}^*$ with the bracket
\begin{align*}
[(x, \alpha), (y, \beta)] = ([x, y], \mathrm{ad}^*_x \beta - \mathrm{ad}^*_y \alpha).
\end{align*}
The direct sum $\mathfrak{g} \oplus \mathfrak{g}^*$ also carries a non-degenerate inner product given by
$\langle (x, \alpha), (y, \beta) \rangle = \frac{1}{2} ( \alpha (y) + \beta (x)).$
\begin{defn}
A generalized complex structure on $\mathfrak{g}$ consists of a linear map $J : \mathfrak{g} \oplus \mathfrak{g}^* \rightarrow \mathfrak{g} \oplus \mathfrak{g}^*$ satisfying
\begin{itemize}
\item[(0)] orthogonality: ~~ $\langle Ju, Jv \rangle = \langle u, v \rangle,$
\item[(i)] almost complex: ~~ $J^2 = - \mathrm{id}$,
\item[(ii)] integrability: ~~
$[Ju, Jv ] - [u, v] - J ( [Ju, v] + [u, Jv]) = 0, ~\text{for } u, v \in \mathfrak{g} \oplus \mathfrak{g}^*.$
\end{itemize}
\end{defn}

Note that the orthogonality condition (0) implies that $J$ must be of the form 
\begin{align}\label{gcs-l}
J = \left( \begin{array}{cc}
N    &      {\bf r}^\sharp \\
\sigma_\flat    &    - N^*
\end{array}
\right),
\end{align}
for some $N \in \mathrm{End}(\mathfrak{g}), ~ {\bf r} \in \wedge^2 \mathfrak{g}$ and $\sigma \in \wedge^2 \mathfrak{g}^*$. However the conditions (i) and (ii) of the definition imposes some relations between structure components of $J$ which are listed in \cite{crainic}.

Thus, a generalized complex structure on $\mathfrak{g}$ can also be considered as a triple $(N, {\bf r}, \sigma)$ such that the linear map $J$ of the form (\ref{gcs-l}) is almost complex and satisfies the integrability condition.
If a Lie algebra $\mathfrak{g}$ admits a generalized complex structure then $\mathfrak{g}$ must be even dimensional. See \cite{gual} for the argument.

\begin{prop}\label{gcs-o-gcs-l}
Let $\mathfrak{g}$ be a Lie algebra. A triple $(N, {\bf r}, \sigma)$  is a generalized complex structure on $\mathfrak{g}$ if and only if the linear map $J = \left( \begin{array}{cc}
N    &      {\bf r}^\sharp \\
\sigma_\flat    &    - N^*
\end{array}
\right)$ is a generalized complex structure on the coadjoint module $\mathfrak{g}^*$ over the Lie algebra $\mathfrak{g}$.
\end{prop}






\section{Holomorphic r-matrices and holomorphic $\mathcal{O}$-operators}\label{sec-holo-o}
Let $\mathfrak{g}$ be a Lie algebra and $J : \mathfrak{g} \rightarrow \mathfrak{g}$ be a complex structure on $\mathfrak{g}$. Consider the complexified vector space $\mathfrak{g}_\mathbb{C} = \mathfrak{g} \otimes \mathbb{C}$. The map $J$ extends to $\mathfrak{g}_\mathbb{C}$ linearly by
$J ( x \otimes c ) = J(x) \otimes c$, for $x \in \mathfrak{g}$, $c \in \mathbb{C}$. Note that $J$ satisfies $J^2 = - \mathrm{id}$. Hence it has eigen values $\pm i$. The corresponding eigen spaces are
\begin{align*}
&(+i)\text{-eigen space} = \mathfrak{g}^{(1,0)} = \{ v \in \mathfrak{g}_\mathbb{C} ~|~ J(v) = + i v \} = \{ x \otimes 1 - Jx \otimes i ~|~ x \in \mathfrak{g} \}, \\
&(-i)\text{-eigen space} = \mathfrak{g}^{(0, 1)} = \{ v \in \mathfrak{g}_\mathbb{C} ~|~ J(v) = - i v \} = \{ x \otimes 1 + Jx \otimes i ~|~ x \in \mathfrak{g} \}.
\end{align*}
Note that the Lie bracket on $\mathfrak{g}$ induce induce Lie brackets on both $\mathfrak{g}^{(1,0)}$ and $\mathfrak{g}^{(0,1)}$.

\begin{defn}
A holomorphic r-matrix on a complex Lie algebra $(\mathfrak{g}, J)$ is a holomorphic bisection ${\bf r }$ (i.e. ${\bf r } \in \wedge^2 \mathfrak{g}^{(1,0)}$ such that $\overline{\partial} {\bf r} = 0 )$ satisfying $[{\bf r} , {\bf r} ] = 0$.
\end{defn}

Since $\wedge^2 \mathfrak{g}_\mathbb{C} = \wedge^2 \mathfrak{g} + i \wedge^2 \mathfrak{g}$, we may write any element ${\bf r} \in \wedge^2 \mathfrak{g}_\mathbb{C}$ by ${\bf r} = {\bf r}_R + i~ {\bf r }_I$. Here ${\bf r}_R$ and  ${\bf r }_I$ are bisections in the real Lie algebra $\mathfrak{g}$ by forgetting the complex structure. Then it has been shown in \cite{geng-st-xu} that ${\bf r } \in \wedge^2 \mathfrak{g}^{(1,0)}$ if and only if ${\bf r}_R^\sharp = {\bf r}_I^\sharp \circ J^*.$ Moreover, the proves the following.

\begin{thm}\cite[Theorem 2.7]{geng-st-xu}\label{holo-thm}
Let $(\mathfrak{g}, J)$ be a complex Lie algebra. Then the followings are equivalent:
\begin{itemize}
\item[(i)] ${\bf r } = {\bf r}_R + i~ {\bf r }_I \in \wedge^2 \mathfrak{g}^{(1,0)}$ is a holomorphic r-matrix,
\item[(ii)] $({\bf r}_I, J)$ is a PN-structure on $\mathfrak{g}$ and ${\bf r }_R^\sharp = {\bf r}^\sharp \circ J^*,$
\item[(iii)] the endomorphism $J = \left( \begin{array}{cc}
J   &      {\bf r}_I^\sharp \\
0    &    -J^*
\end{array}
\right)$ is a generalized complex structure on $\mathfrak{g}$ and ${\bf r }_R^\sharp = {\bf r}^\sharp \circ J^*.$
\end{itemize} 
\end{thm} 

The above characterizations of holomorphic r-matrices allows us to introduce holomorphic $\mathcal{O}$-operators.
Let $(\mathfrak{g}, J)$ be a complex Lie algebra and $(M, J_M)$ be a representation over it.

\begin{defn}
A holomorphic $\mathcal{O}$-operator on $M$ over the complex Lie algebra $(\mathfrak{g}, J)$ consists of a pair of linear maps $(T_R, T_I) : M \rightarrow \mathfrak{g}$ satisfying the properties that $(T_I, J, J_M)$ is an $\mathcal{ON}$-structure on $M$ over the Lie algebra $\mathfrak{g}$ and $T_R = T_I \circ J_M$.
\end{defn}

\begin{remark}
It follows from the above definition that both $T_R$ and $T_I$ are $\mathcal{O}$-operators in real sense and they are related by $T_R = T_I \circ J_M = J \circ T_I$.
\end{remark}

Finally, by Theorem \ref{holo-thm}, we have following.
\begin{prop}
Let $(\mathfrak{g}, J)$ be a complex Lie algebra. Then ${\bf r } = {\bf r}_R + i~ {\bf r}_I$ is a holomorphic r-matrix if and only if $({\bf r}_R^\sharp , {\bf r}_I^\sharp )$ is a holomorphic $\mathcal{O}$-operator on $\mathfrak{g}^*$ over the complex Lie algebra $(\mathfrak{g}, J).$
\end{prop}







\medskip

\noindent {\em Acknowledgements.} The research is supported by the fellowship of Indian Institute of Technology (IIT) Kanpur. The author thanks the Institute for support.

\end{document}